\documentclass[11pt]{article}

\def\UseSection{
        \numberwithin{equation}{section}
        \newtheorem{theorem}    {Theorem}[section]
        \DefineTheorems 
}
\usepackage[textwidth=500pt,textheight=660pt,centering]{geometry} 

\usepackage{amsfonts}
\usepackage{amsmath,amssymb,amsthm}
\usepackage{appendix}
\usepackage{bbm} 
\usepackage{amsbsy}
\usepackage{enumerate}
\usepackage{cite}

\usepackage[bookmarks, colorlinks, breaklinks,
pdftitle={},pdfauthor={}]{hyperref}

\hypersetup{
  linkcolor=black,
  citecolor=black,
  filecolor=black,
  urlcolor=black}

\usepackage{verbatim}
\usepackage{tikz}
\usepackage{color}

\newcommand{\black}{\black}

\usepackage[font=small,labelfont=bf]{caption}

\numberwithin{equation}{section}

\newcommand{\bb}[1]{\mathbb{#1}}

\newcommand{\blank}[1]{}

\newcommand{\R}{\bb R}
\newcommand{\Z}{\bb Z}
\newcommand{\C}{\bb C}

\newcommand{\N}{\bb N}
\newcommand{\T}{\bb T}

\newcommand{\nnb}	{\nonumber \\}

\def\DefineTheorems{
	\newtheorem{lemma}      [theorem] {Lemma}

	\theoremstyle{definition}
	
	\newtheorem{rk}       [theorem] {Remark}
    \newtheorem{example}       [theorem] {Example}
}

\UseSection
\newcommand{\lbeq}[1]  {\label{e:#1}}
\newcommand{\refeq}[1] {\eqref{e:#1}}    

\title{Kotani's Theorem for the Fourier Transform}
\author{Gordon Slade\thanks{Department of Mathematics,
     University of British Columbia,
     Vancouver, BC, Canada V6T 1Z2.
     https://orcid.org/0000-0001-9389-9497.
     E-mail: {{\tt slade@math.ubc.ca}}}}

\begin{document}

\date{\vspace{-5ex}} 

\maketitle

\begin{abstract}
In 1991, Shinichi Kotani proved a theorem giving
a sufficient condition to conclude that a function $f(x)$
on $\Z^d$ decays like $|x|^{-(d-2)}$ for large $x$, assuming that its Fourier transform
$\hat f(k)$
is such that $|k|^{2}\hat f(k)$ is well behaved for $k$ near zero.
The proof was not published.
We prove an extension of Kotani's Theorem, based on Kotani's unpublished proof.
\end{abstract}



\section{Introduction and main result}

Our purpose here is to revive and extend a neglected theorem of Shinichi Kotani that was stated
without its proof as \cite[Theorem~1.6.1]{MS93}.
The theorem  gives
a sufficient condition to conclude that a function $f(x)$
on $\Z^d$ decays like $|x|^{-(d-2)}$ for large $x$, assuming that its Fourier transform
$\hat f(k)$
is such that $|k|^{2}\hat f(k)$ is well behaved for $k$ near zero.
We state Kotani's Theorem and prove an extension.
Our proof is adapted from
Kotani's proof of \cite[Theorem~1.6.1]{MS93}
in his unpublished 1991 handwritten notes in Japanese
\cite{Kota91}, which were translated into English at that time by Takashi Hara.

Let $\mathbb{T}^d = (\R / 2 \pi \Z )^d$ denote the torus.
For a summable function $f:\Z^d \to \C$ we define its Fourier transform by
\begin{equation}
    \hat f(k) = \sum_{x\in\Z^d}f(x)e^{ik\cdot x}
    \qquad
    (k \in \T^d).
\end{equation}
The inverse Fourier transform is
\begin{equation}
    f(x) = \int_{\T^d}\hat{f}(k) e^{-i k \cdot x } \frac{dk}{(2\pi)^{d}}
    \qquad
    (x \in \Z^d).
\end{equation}
For $d>2$, we define the dimension-dependent constants
\begin{equation}
\lbeq{nddef}
    a_d=\frac{\Gamma(\frac{d-2}{2})}{4 \pi^{d/2}},
    \qquad
    n_d= \begin{cases}
    d-2 & (d>4)
    \\
    d-1 & (d=3,4).
    \end{cases}
\end{equation}
For an open set $U\subset \T^d$, $C^n(U)$ denotes the space of functions $\hat f:U \to \C$
with derivatives $\nabla^\alpha \hat f$ continuous for all $|\alpha| \le n$.
The following is Kotani's Theorem in its original form, as stated in \cite[Theorem~1.6.1]{MS93}.
Kotani's proof was not published.

\begin{theorem}
\label{thm:kotani-original}
Let $d >2$.
Given $\hat{f} \in C^{d-2} (\mathbb{T}^d \backslash \{0\})$, let $\hat{h}(k)=|k|^2\hat{f}(k)$.
Suppose that there is a neighbourhood  $0 \in U \subset \mathbb{T}^d$ such that
$\hat{h} \in  	C^{n_d}(U)$.
Then as $|x| \to \infty$,
\begin{equation}
\lbeq{gxasy-original}
	f(x) =
	a_d \hat{h}(0) \frac{1}{|x|^{d-2}} + o\left(\frac{1}{|x|^{d-2}} \right)  .
\end{equation}
\end{theorem}

In \cite[Example~1.6.2]{MS93}, it is shown that the hypothesis of existence of
at least $d-2$
derivatives for $\hat{h}$ cannot be relaxed.  Namely, the
example exhibits a function $\hat{f}$ on $\mathbb{T}^d$, for $d \ge 3$,
with $\hat{h}(k) =|k|^2 \hat{f}(k)$ having $d-3$ but not
$d-2$ derivatives in a neighbourhood of $k=0$, and with $f(x)$
not bounded above by a multiple of $|x|^{-(d-2)}$ for large $x$.

\begin{example}[Random walk Green function]
\label{ex:srw}
Let $d>2$.
Let $D$ be the one-step transition function of a random walk
on $\Z^d$, i.e., $D:\Z^d\to [0,1]$ with $\sum_{x\in\Z^d} D(x)=1$.
Suppose that $D$ obeys $\Z^d$-symmetry and has moments up to and including order
$n_d+2$.
Then $\hat D(k)$ has a Taylor expansion about $k=0$ to order $n_d+2$.
The Green function $C(x)$ is the inverse Fourier transform of $\hat C(k) = \frac{1}{1-\hat D(k)}$.
Let $\sigma^2 = \sum_x |x|^2 D(x)$.
Then $\hat h(k) = |k|^2 \hat C(k)$
obeys $\hat h(0)=2d\sigma^{-2}$.  Moreover, $1/\hat h(k) = |k|^{-2}[1-\hat D(k)]$
has an expansion about $k=0$ to order $n_d$, with constant term $\sigma^2/(2d)$,
so $\hat h\in C^{n_d}(\T^d)$.  It then follows from Theorem~\ref{thm:kotani-original}
 that
\begin{equation}
\lbeq{RWasy}
    C(x) = a_d\frac{2d}{\sigma^2}\frac{1}{|x|^{d-2}}
    + o\left(\frac{1}{|x|^{d-2}} \right).
\end{equation}
This provides an alternate approach to obtaining leading Green function asymptotics,
compared to, e.g., \cite{Uchi98,LL12}.  Fourier analysis is also used in \cite{Uchi98,LL12}, but
the analysis in the proof of Theorem~\ref{thm:kotani-original} is relatively mild.
\end{example}

\begin{rk}
The hypothesis of existence of $n_d+2$ moments in Example~\ref{ex:srw} is
more restrictive than
the hypothesis in \cite[Theorem~2]{Uchi98}  to conclude \refeq{RWasy}, which requires that
$D$ has at least $d-2$ moments for $d>4$,
at least two moments for $d=3$,  while for
$d=4$ the requirement is $\sum_x |x|^2 (\log |x|) D(x)<\infty$.
However, unlike \cite[Theorem~2]{Uchi98}, Theorem~\ref{thm:kotani-original} allows negative
values for $D(x)$.
The hypothesis in Theorem~\ref{thm:kotani-original} differs from the one in \cite[Theorem~1.4]{Hara08}
(see also \cite[Section~7.1]{Uchi98}),
which concerns functions of the form $\hat f(k) = [1-\hat J(k)]^{-1}$ with $J$ subject
to certain hypotheses in $x$-space including the relatively mild estimate
$|J(x)|\le O(|x|^{-(d+2+\epsilon)})$ with $\epsilon \ge 0$, again without
requiring $J(x) \ge 0$.  The proof of \cite[Theorem~1.4]{Hara08} is considerably more
involved than the proof of Theorem~\ref{thm:kotani}
presented here.
\end{rk}

The next theorem is an extension of Theorem~\ref{thm:kotani-original}
with an explicit error estimate and a weaker hypothesis on
$\hat h$.
For $p \in [1,\infty)$, we use the norms
\begin{equation}
    \|\hat f\|_{L^p}^p = \int_{\T^d}|\hat f(k)|^p\frac{dk}{(2\pi)^d}  .
\end{equation}
For $n\in\N$, the Sobolev space $W^{n,p}$ denotes the space of functions on $\T^d$ which have all
derivatives of order less than or equal to $n$ in $L^p(\T^d)$, with norm
\begin{equation}
    \|\hat f\|_{W^{n,p}} = \sum_{|\alpha| \le n} \|\nabla^\alpha \hat f\|_{L^p}.
\end{equation}

\begin{theorem}
\label{thm:kotani}
Let $d > 2$.  Choose $p_d \in (\frac{d}{d-2},2]$ for $d>4$, and set $p_d=2$ for $d=3,4$.
Let $U,V$ be open subsets of $\T^d$ with $0 \in V \subset \overline{V} \subset U$.
Let $\hat f$ be a continuous function on $\T^d \setminus \{0\}$ such that $\hat f \in W^{d-2,1}(\T^d\setminus V)$.
Let $\hat{h}(k)=|k|^2\hat{f}(k)$ be continuous on $\T^d$
and suppose that $\hat{h} \in W^{n_d,p_d}(U)$.
Then
\begin{equation}
\lbeq{gxasy}
	f(x) =
	\frac{a_d}{|x|^{d-2}} \left(  \hat{h}(0) + R_f(x)\right)
\end{equation}
with $\lim_{x \to \infty}R_f(x)=0$.
 Moreover, $R_f$ obeys the upper bound
\begin{equation}
\lbeq{Rbound}
    |R_f(x)| \le
    c
    \left(
    \|\hat f  \|_{W^{d-2,1}(\T^d \setminus V)}
    +
    \|  \hat h   \|_{W^{n_d,p_d}(U)}
    \right)
    ,
\end{equation}
with constant $c$ depending on $U,V,d$.
\end{theorem}

The bound \refeq{Rbound} does not decay in $x$ but it does provide a bound
$|f(x)| \le c_1 |x|^{-(d-2)}$ with explicit dependence of $c_1$ on $\hat f$.
Note that in \refeq{Rbound} the norm of $\hat f$
is on $\T^d \setminus V$, so the possibly singular behaviour of $\hat f$ at the origin
does not contribute.  Also, the hypothesis on $\hat h$ permits singular behaviour of
derivatives of $\hat h$.\footnote{An application of Theorem~\ref{thm:kotani} to the lace expansion
is given in Version~1 of this arXiv
submission; subsequently a more direct proof was discovered
that does not require Theorem~\ref{thm:kotani} for that purpose \cite{Slad20_lace}.}

\section{Proof of Theorem~\ref{thm:kotani}}

Our proof of Theorem~\ref{thm:kotani} is adapted from
Kotani's unpublished proof of Theorem~\ref{thm:kotani-original}.
The proof  uses the Fourier transform
\begin{equation}
    \hat f(k) = \int_{\R^d} f(y) e^{ik\cdot y} dy
    \qquad
    (k \in \R^d)
\end{equation}
of suitable functions $f:\R^d \to \C$, with inverse Fourier transform
\begin{equation}
\lbeq{RdinvFT}
    f(y) = \int_{\R^d} \hat f(k) e^{-ik\cdot y} \frac{dk}{(2\pi)^d}
    \qquad
    (y \in \R^d).
\end{equation}
Two items in the proof are deferred to two lemmas that follow the proof.

\begin{proof}[Proof of Theorem~\ref{thm:kotani}]
Let $\hat\chi \in C^\infty (\T^d)$
be a fixed bump function supported in $U$, with  $\hat\chi(k) =1$
for $k\in V$ and $\hat\chi (k) \in [0,1]$ for all $k$.
We set $\hat{g}(k)=\hat f(k)(1-\hat\chi(k))$, and make the decomposition
\begin{align}
    f(x) &
    = \int_{\T^d} \hat f(k)  \hat\chi(k) e^{-ik\cdot x} \frac{dk}{(2\pi)^d}
    + \int_{\T^d} \hat g(k)  e^{-ik\cdot x} \frac{dk}{(2\pi)^d}
    \nnb &
    =
    I_1(x)+I_2(x).
\end{align}
By hypothesis, $\nabla^{\alpha} \hat g \in L^1(\T^d)$ for $|\alpha| \le d-2$,
so by the
Riemann--Lebesgue Lemma (and integration by parts)
\begin{equation}
    I_2(x) = o(|x|^{-(d-2)}).
\end{equation}
There is also the quantitative estimate, for any component $i$,
\begin{equation}
    |I_2(x)| \le \frac{1}{x_i^{d-2}} \| \nabla^{d-2}_i \hat g\|_{L^1(\T^d)}.
\end{equation}
The $C^\infty$ function $\hat\chi$ depends only on the neighbourhoods $U,V$ and does not depend on $\hat f$.  There is therefore a constant depending only on $d,U,V$ (and the choice of $\hat\chi$) such that
\begin{equation}
    |I_2(x)| \le \frac{c}{|x|^{d-2}} \|  \hat f   \|_{W^{d-2,1}(\T^d \setminus V)}.
\end{equation}
Thus $I_2(x)$ obeys the desired error estimate and
the main term is $I_1(x)$.

We can extend the integral defining $I_1(x)$ to all of $\R^d$.
We define a function $\hat s$ on $\R^d$ by
\begin{equation}
    \hat s(k) =
    \begin{cases}
        \hat h(k) \hat\chi(k) & (k\in U)
        \\
        0 & (k \in \R^d\setminus U).
    \end{cases}
\end{equation}
By Lemma~\ref{lem:s},
the inverse Fourier transform $s$ of $\hat s$, defined by \refeq{RdinvFT}, satisfies
$s \in L^t(\R^d)$ for $t \in [1,\infty)$ and
$\|s\|_{L^1(\R^d)} \le c \|\hat h\|_{W^{n_d,p_d}(U)}$.
By definition, and by the fact\footnote{The
identity \refeq{I1} is well known when $\hat s$ is
a Schwartz function, e.g., \cite[Theorem~2.4.6]{Graf10}.
Our $\hat s$, although continuous and of compact support, need not be a Schwartz function.
The following argument shows that \refeq{I1}
holds nevertheless.  Fix a $C^\infty$ function
$\hat\psi:\R^d\to [0,\infty)$ with compact support,
such that $\hat\psi_\epsilon(k)=\epsilon^{-d}\hat\psi(k/\epsilon)$
is an approximate identity (as in \cite[Definition~1.2.15]{Graf10}).  In particular,
$\|\hat\psi_\epsilon\|_{L^1}=1$.
Let $\hat s_\epsilon(k)=(\hat s*\hat\psi_\epsilon)(k)$; then $\hat s_\epsilon$ is a Schwartz function so
\refeq{I1} holds with $s$ replaced by $s_\epsilon$.
Also, $\|\hat{s}_\epsilon\|_{L^\infty} \le \|\hat s\|_{L^\infty} < \infty$,
and $|s_\epsilon| = |s\psi_\epsilon| \le |s|$ since $\psi_\epsilon(y) \le \|\hat\psi_\epsilon\|_{L^1}=1$.
Then we obtain \refeq{I1} by letting
$\epsilon \to  0$, using $\hat s_\epsilon(k) \to \hat s(k)$,
$\psi_\epsilon (y) \to 1$, and dominated convergence.}
that $|k|^{-2}$ is the Fourier transform
of $a_d|y|^{-(d-2)}$,
\begin{equation}
\lbeq{I1}
    I_1(x) = \int_{\R^d} \frac{\hat s(k)}{|k|^2} e^{-ik\cdot x} \frac{dk}{(2\pi)^d}
    =
    a_d \int_{\R^d} \frac{1}{|x-y|^{d-2}}s(y)dy.
\end{equation}
Given $\epsilon \in (0,1)$, we further decompose
\begin{align}
    I_1(x)
    &=
    a_d \int_{|y|<\epsilon |x|} \frac{1}{|x-y|^{d-2}}s(y)dy
    +
    a_d \int_{|y|\ge \epsilon |x|} \frac{1}{|x-y|^{d-2}}s(y)dy
    \nnb & =
    J_1(x) + J_2(x).
\end{align}

Since $s\in L^1(\R^d)$, the definition of $\hat h$ gives
\begin{equation}
    \hat h(0) = \hat s(0) = \int_{\R^d} s(y) dy.
\end{equation}
Therefore,
\begin{align}
    |x|^{d-2}I_1(x) -  a_d \hat h(0)
    &=
    |x|^{d-2}J_1(x) + |x|^{d-2}J_2(x) -  a_d \int_{\R^d}s(y)dy
    \nnb & =
    |x|^{d-2}J_2(x) -
    a_d \int_{|y|<\epsilon |x|} \left(1- \frac{|x|^{d-2}}{|x-y|^{d-2}}  \right) s(y)dy
    - a_d \int_{|y|\ge \epsilon |x|}s(y)dy.
\end{align}
By Lemma~\ref{lem:J2},  $|x|^{d-2}J_2(x)=o(1)$ and also
$J_2(x) = O(\|\hat h\|_{W^{n_d,p_d}(U)})$
as in \refeq{Rbound}.
By Lemma~\ref{lem:s}, the final term is also $o(1)$ and at most
$O(\|\hat h\|_{W^{n_d,p_d}(U)})$.
For the middle term, we write $\hat x = x/|x|$ and
use the fact that when $|y|<\epsilon |x|$ (now we take $\epsilon$ small) we have
\begin{equation}
    \left| 1 - \frac{|x|^{d-2}}{|x-y|^{d-2}} \right|
    =
    \left| 1 - \frac{1}{|\hat x- y/|x||^{d-2}} \right|
    \in
    \left( 1-\frac{1}{(1-\epsilon)^{d-2}},1 - \frac{1}{(1+\epsilon)^{d-2}} \right).
\end{equation}
Therefore, since $s \in L^1(\R^d)$, the second term is $o(1)$ by dominated convergence,
and it is $O(\|\hat h\|_{W^{n_d,p_d}(U)})$ since $\|s\|_{L^1(\R^d)}$
is by Lemma~\ref{lem:s}.

This completes the proof, subject to the following two lemmas.
\end{proof}

We write $n$ in place of $n_d$ (recall \refeq{nddef}),
$p$ in place of $p_d$ (specified in Theorem~\ref{thm:kotani}), and
let $q$ be the dual index defined by
$\frac 1q + \frac 1p =1$.
Both lemmas use the function
\begin{equation}
    u(y)=|y|^{n }s(y) \qquad (y \in \R^d).
\end{equation}

\begin{lemma}
\label{lem:s}
Under the assumptions of Theorem~\ref{thm:kotani},
$\|u\|_{L^{q}(\R^d)} \le c\|\hat h\|_{W^{n,p}(U)}$,
$\lim_{y\to\infty}u(y)=0$,
$|u(y)| \le c \|\hat h\|_{W^{n,p}(U)}$,
$\|s\|_{L^1(\R^d)} \le c \|\hat h\|_{W^{n,p}(U)}$,
and $s\in L^t(\R^d)$ for all $t \in [1,\infty)$.
\end{lemma}

\begin{proof}
By hypothesis, $\nabla_i^{n}\hat s \in L^p(\R^d)$.  Its inverse Fourier transform
is essentially $y_i^{n}s(y)$.
By H\"older's inequality, $|y|^n \le d^{n-2} \sum_{i=1}^d |y_i|^n$.
Since $p \in (1,2]$ by hypothesis, we can apply
the Hausdorff--Young inequality to see that
(the constant $c$ can change from one occurrence to the next)
\begin{equation}
\lbeq{ubd1}
    \|u\|_{L^q(\R^d)} \le c\max_i \|y_i^{n}s(y)\|_{L^q(\R^d)}
    \le c\max_i \|\nabla_i^{n}\hat s\|_{L^p(U)}
    \le c  \| \hat h\|_{W^{n,p}(U)}
    .
\end{equation}
Also,
\begin{equation}
    |u(y)| \le c\max_i |y_i^{n} s(y)|
    = c\max_i \Big|\int_{\R^d} \nabla_i^{n}\hat s(k) e^{-ik\cdot y} \frac{dk}{(2\pi)^d}\Big|.
\end{equation}
Because $\nabla_i^{n}\hat s$ has compact support and is in $L^p$, it is also in $L^1$,
so the right-hand side goes to zero as $|y|\to\infty$, by the Riemann--Lebesgue Lemma.
It is also bounded above by a multiple of the $L^p$ norm controlled in \refeq{ubd1}.
This proves the three statements for $u$.

Since $\hat s$ is continuous and has compact support,
 $\hat s \in L^t(\R^d)$ for all $t \in [1,\infty)$.  Therefore,
by the Hausdorff--Young inequality, $s \in L^r(\R^d)$ for all $r \ge 2$.
It follows that
\begin{equation}
    \int_{|y|\le 1} |s(y)|dy \le c \left( \int_{|y|\le 1} |s(y)|^q dy \right)^{1/q}
    \le c \|s\|_{L^q(\R^d)} \le c \|\hat s\|_{L^p(\R^d)}
    \le
    c\|\hat h\|_{W^{d-2,p}(U)}.
\end{equation}
Also, by H\"older's inequality,
\begin{align}
    \int_{|y|> 1}|s(y)|dy & \le
    \left( \int_{|y|> 1}\frac{1}{|y|^{pn}} dy \right)^{1/p}
    \left( \int_{|y|> 1}|u(y)|^q dy \right)^{1/q}  .
\end{align}
For $d>4$ we have $pn =p(d-2) > d$ by hypothesis, and for $d=3,4$ we have $pn = 2(d-1)>d$.
Therefore the first integral is bounded.  The second is bounded by
$\|u\|_{L^q(\R^d)}$, whose upper bound we have already discussed.  This proves that
$s \in L^1(\R^d)$, with the desired estimate.  Since also $s \in L^t(\R^d)$ for all $t \ge 2$,
in particular for $t=2$, it follows that $s \in L^t(\R^d)$ for all $t \ge 1$.
\end{proof}

\begin{lemma}
\label{lem:J2}
Under the assumptions of Theorem~\ref{thm:kotani},
$|x|^{d-2}J_2(x)$ is both $o(1)$ and $O(\|\hat h\|_{W^{n,p}(U)})$.
\end{lemma}

\begin{proof}
With
$A_{21}= \{|y|\ge \epsilon |x|\} \cap \{|x-y|\ge 1\}$ and
$A_{22}= \{|y|\ge \epsilon |x|\} \cap \{|x-y|< 1\}$,
we make the decomposition
\begin{align}
    J_2(x) & =
    a_d \int_{A_{21}} \frac{1}{|x-y|^{d-2}}s(y)dy
    +
    a_d \int_{A_{22}} \frac{1}{|x-y|^{d-2}}s(y)dy
    \nnb & =
    J_{21}(x)+J_{22}(x).
\end{align}

By H\"older's inequality,
\begin{align}
\lbeq{J21}
    |J_{21}(x)| & \le
    a_d \left( \int_{A_{21}}\frac{1}{|x-y|^{p(d-2)}}\frac{1}{|y|^{pn}} dy \right)^{1/p}
    \left( \int_{A_{21}}|u(y)|^q dy \right)^{1/q} .
\end{align}
For $d>4$, since
$pn=p(d-2)>d$ by hypothesis,
\begin{equation}
    \left( \int_{A_{21}}\frac{1}{|x-y|^{p(d-2)}}\frac{1}{|y|^{p(d-2)}} dy \right)^{1/p}
    \le \frac{1}{(\epsilon |x|)^{d-2}}
    \left( \int_{|x-y| \ge 1}\frac{1}{|x-y|^{p(d-2)}}  dy \right)^{1/p}
    \le \frac{c}{ (\epsilon |x|)^{d-2}}.
\end{equation}
For $d=3,4$, since $pn=2(d-1)$, we use instead
\begin{align}
    \left(  \int_{A_{21}}\frac{1}{|x-y|^{2(d-2)}|y|^{2(d-1)}} dy \right)^{1/2}
    & \le
    \left( \int_{|x-y|\ge 1} \frac{1}{|x-y|^{4(d-2)}} dy\right)^{1/4}
    \left( \int_{|y|\ge \epsilon |x|} \frac{1}{|y|^{4(d-1)}}dy \right)^{1/4},
\end{align}
and observe that
the first factor is finite since $4(d-2)>d$ for $d=3,4$,
and that since $4(d-1)>d$ the
second is $O(|x|^{-\frac 14 (3d-4)}) = O(|x|^{-(d-2)})$.
For any $d>2$, the integral in the second factor of \refeq{J21} obeys
\begin{equation}
\lbeq{A21u}
    \int_{A_{21}}|u(y)|^q dy
    \le
    \int_{|y|\ge \epsilon |x|}|u(y)|^q dy .
\end{equation}
Since $u \in L^q(\R^d)$ by Lemma~\ref{lem:s},
the right-hand side of \refeq{A21u} is $o(1)$ as $x\to\infty$, and
also its $q^{\rm th}$ root is bounded by $\|u\|_{L^q(\R^d)}$ which
at most $O(\|\hat h\|_{W^{n,p}(U)})$ by Lemma~\ref{lem:s}.

Finally,
\begin{equation}
    J_{22}(x)
    \le
    a_d \frac{1}{(\epsilon |x|)^{n}}
    \left(\sup_{|y| \ge \epsilon |x|}|u(y)| \right)
    \int_{|x-y| < 1} \frac{1}{|x-y|^{d-2}}dy
    \le
    c_\epsilon \frac{1}{ |x|^{n}}
    \left(\sup_{|y| \ge \epsilon |x|}|u(y)| \right),
\end{equation}
since the $y$-integral is finite.
For $d>4$ the power on the right-hand side is $|x|^{-(d-2)}$, and for $d=3,4$ it
is $|x|^{-(d-1)} = o(|x|^{-(d-2)})$.
By Lemma~\ref{lem:s}, the supremum factor  is $o(1)$
since $u(y) \to 0$ as $y \to \infty$, and also is at most $O(\|\hat h\|_{W^{n,p}(U)})$.
This completes the proof.
\end{proof}

\section*{Acknowledgements}

This work was supported in part by NSERC of Canada.
I am grateful to Shinichi Kotani for providing his unpublished notes in 1991
with the proof of Theorem~\ref{thm:kotani-original}
that forms the basis for the proof of Theorem~\ref{thm:kotani}
and for recent correspondence
concerning this work,
and to Takashi Hara for his translation of Professor Kotani's notes in  1991.


\end{document}